\documentclass[11pt,a4paper]{article}
\usepackage{lmodern}
\usepackage{babel}
\usepackage[T1]{fontenc} 
\usepackage[utf8]{inputenc} 
\usepackage{verbatim}  
\usepackage[a4paper]{geometry} 
\usepackage{amsmath} 
\usepackage{graphicx} 
\graphicspath{{images/}} 
\usepackage{systeme}
\usepackage{amssymb}
\usepackage{amsfonts}
\usepackage{csquotes}
\usepackage{hyperref}
\usepackage{algorithm}
\usepackage{algpseudocode}

\usepackage{amsmath,amssymb}
\usepackage{amsthm}

\newtheorem{theorem}{Theorem}[section]
\newtheorem{conjecture}{Conjecture}[section]
\newtheorem{definition}{Definition}[section]

\newtheorem{example}{Example}[section]

\usepackage{amssymb}
\usepackage{amsmath}
\usepackage{amstext}

\usepackage{comment}

\newcommand*\modd[3]{#1\equiv#2\;( {\rm mod}\;  #3)}
\newcommand*\notmodd[3]{#1\not\equiv#2\;( {\rm mod}\;  #3)}

\title{An Extension of the Collatz Conjecture modulo  $2^p+2^q$}

\author{Abderrahman  Bouhamidi\footnote{L.M.P.A, Universit\'e du
		Littoral, 50 rue F. Buisson BP699, F-62228 Calais-Cedex, France.  \href{mailto:abderrahman.bouhamidi@univ-littoral.fr}{abderrahman.bouhamidi@univ-littoral.fr}}}

\date{November 5,  2025}
\begin{document}
\maketitle

\abstract{In this paper, we will introduce  an extension to the Collatz's conjecture.  This  conjecture  may be seen as a general conjecture that unifies   the Collatz one together with many other similar conjectures.  For instance, we propose our new conjecture modulo $10$ which may be stated   as follows.  Starting from any  positive integer,
	if it  is a multiple of $10$   then divide it by 10, otherwise,  multiply it by $12$, add $8$ times  its  last digit and divide the result  by $10$. Repeat the process infinitely. Regardless the starting number, the process eventually reaches $4$  after a finite number of iterations.  The genaral conjecture studied here will encompasse  the classical Collatz conjecture togher  with our proposed one modulo $10$.
}




\maketitle

\section{Introduction}\label{sec1}
The Collatz conjecture (stated since $1937$) is one of the famous unsolved problem in mathematics.   This conjecture is also  well known as  the $3n+1$ problem and  is formulated as follows: Starting from any initial integer  $\geq 1$, if the  number  is even divide it by $2$ else multiply it by $3$ and add $1$.  The Collatz conjecture asserts  that repeating  the process iteratively, it always reaches   $1$.  A survey and extensive of literature  on this subject has been given  in \cite{Lagarias1985,Lagarias2011}.  Consider the mapping $C:\mathbb{N}:\longrightarrow \mathbb{N}$ given by
\begin{equation}\label{CollatzMap}
	C(n)=\left\{\begin{array}{lcc}
		n/2&\mathrm{if}& \modd{n}{0}{2}\\ 
		(3n+1)/2&\mathrm{if}& \notmodd{n}{0}{2},\\
	\end{array}\right.
\end{equation}
where $\mathbb{N}=\{1,2,3,\ldots\}$.
The Collatz conjecture may be reformulated as following: For any $n\in\mathbb{N}$, there is a positive integer $k\in\mathbb{N}_0:=\mathbb{N}\cup\{0\}$ such that  $C^{(k)}(n)=1$, where the notation $C^{(k)}$ stands for the $k$-th iterate of the mapping $C$. 
It is also conjectured that the trivial cycle $(1\rightarrow 2\rightarrow 1)$ is  the unique cycle of the Collatz operator $C$.   The Collatz conjecture has been  verified experimentally with  computer  by many authors. For instance,  in  \cite{Oliviera2010}, the author claims that he verified the Collatz conjecture  in $2009$  up to 
$2^{62}\simeq 4.61\times 10^{18}$ and in  
\cite{Barina2025}, the author claims that he verified    the  conjecture  in $2025$  up to  $2^{71}\simeq 2.36\times 10^{21}$.

Several authors have proposed important generalizations of the Collatz  mapping, see \cite{Allouche1978,Crandall1978,Lagarias1985, MatthewsWatts1984}. Terras \cite{Terras1976} introduced the concept of stopping time. He also proved that almost all integers have finite stopping time, see also \cite{Everett1977}. In \cite{Eliahou1993}, Eliahou obtained  lower bound  lengths of hypothetic cycle. Recently, Tao \cite{Tao2020} established  that almost all orbits of the Collatz map attain almost bounded values,

In this paper, we propose a general  extension of the Collatz conjecture that unifies many Collatz-like maps. We begin by introducing our first simple case conjecture in a manner that is accessible to a broad audience, including non-mathematicians.
A preliminary and modest computational verification has been carried out for the validation of this conjecture; more details are provided in the following section.

\begin{conjecture}\label{Myconjecture}
	Starting with any positive integer  $n\geq 1$:
	\begin{itemize}
		\item[$\bullet$]  If $n$  is a multiple of $10$  then divide the number by $10$.  Otherwise, 
		multiply it by $12$, add  $8$ times  its  last digit and divide the resultby $10$.
		\item Repeat the process infinitely.
	\end{itemize}
	Then,  regardless the starting number, the process eventually reaches $4$  after a finite number of iterations.
\end{conjecture}
The previous conjecture is in fact associated to the operator
$T:~\mathbb{N}~\longrightarrow~\mathbb{N}$ given by
\begin{equation}\label{mapT0}
	T(n)=\left\{\begin{array}{ccc}
		n/10&\mathrm{if}& \modd{n}{0}{10}\\
		(12 n+8 [ n]_{10})/10&\mathrm{if}& \notmodd{n}{0}{10},\\
	\end{array}\right.\end{equation}
or
\begin{equation}\label{mapT01}
	T(n)=\left\{\begin{array}{ccc}
		n/10&\mathrm{if}& \modd{n}{0}{10}\\
		(6 n+4 [ n]_{10})/5
		&\mathrm{if}& \notmodd{n}{0}{10}.\\
	\end{array}\right.
\end{equation}
where the notation $[n]_d$ denotes the remainder of $n$ in the Euclidean division of $n$ by $d\geq 2$ with the standard condition   $0\leq [n]_d<d$. For $d=10$, the remainder $[n]_{10}$ is exactly  the last digit of the number $n$ in the decimal base.  Our conjecture may now be reformulated more precisely as following:  For any $n\in\mathbb{N}^*$, there is an integer $k\in\mathbb{N}$ such that  $T^{(k)}(n)=4$.  The trivial cycle $$(4\rightarrow  8\rightarrow 16\rightarrow  24\rightarrow 32\rightarrow  40\rightarrow 4)$$ is  the unique cycle of the operator $T$. Let us, give two  examples: Starting with the integers $n=75$ and $n=135$ leads to the trajectories

\begin{itemize}
	\item $n = 75\rightarrow 94\rightarrow 116\rightarrow 144\rightarrow 176\rightarrow 216\rightarrow 264\rightarrow 320\rightarrow 32\rightarrow 40\rightarrow \textbf{4} \rightarrow \textbf{8} \rightarrow \textbf{16} \rightarrow \textbf{24} \rightarrow \textbf{32} \rightarrow \textbf{40} \rightarrow \textbf{4}$
	\item $n = 135 \rightarrow 166 \rightarrow 204 \rightarrow 248 \rightarrow 304 \rightarrow 368 \rightarrow 448 \rightarrow 544 \rightarrow 656 \rightarrow 792 \rightarrow 952 \rightarrow 1144 \rightarrow 1376 \rightarrow 1656 \rightarrow 1992 \rightarrow 2392 \rightarrow 2872 \rightarrow 3448 \rightarrow 4144 \rightarrow 4976 \rightarrow 5976 \rightarrow 7176 \rightarrow 8616 \rightarrow 10344 \rightarrow 12416 \rightarrow 14904 \rightarrow 17888 \rightarrow 21472 \rightarrow 25768 \rightarrow 30928 \rightarrow 37120 \rightarrow 3712 \rightarrow 4456 \rightarrow 5352 \rightarrow 6424 \rightarrow 7712 \rightarrow 9256 \rightarrow 11112 \rightarrow 13336 \rightarrow 16008 \rightarrow 19216 \rightarrow 23064 \rightarrow 27680 \rightarrow 2768 \rightarrow 3328 \rightarrow 4000 \rightarrow 400 \rightarrow 40 \rightarrow \textbf{4} \rightarrow \textbf{8} \rightarrow \textbf{16} \rightarrow \textbf{24} \rightarrow \textbf{32} \rightarrow \textbf{40} \rightarrow \textbf{4}$
\end{itemize}
In Figure~\ref{fig_Traj75-135},
the graphs represent the trajectories  starting from $75$ and from $135$, respectively. 

\begin{figure}[h]%
	\centering
	\begin{tabular}{cc}
		\includegraphics[width=5cm,height=5cm]{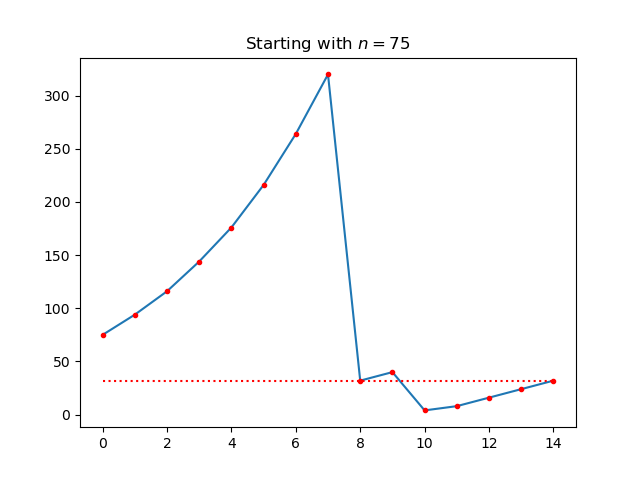}&
		\includegraphics[width=5cm,height=5cm]{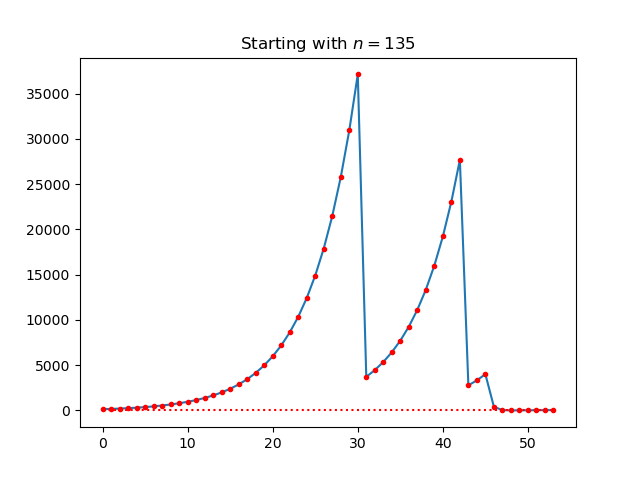}
	\end{tabular}
	\caption{Trajectories : Strating from $n=75$ (left) and from  $n=135$ (right)}\label{fig_Traj75-135},
\end{figure}

We have discussed, in \cite{Bouhamidi_Hal2024},  a general   extension to the classical Collatz conjecture by considering the general following operator $T:\mathbb{N}:\longrightarrow \mathbb{N}$ given by
\begin{equation}\label{mapT}
	T(n)=\left\{\begin{array}{ccc}
		n/d&\mathrm{if}& \modd{n}{0}{d}\\
		(\alpha n+\beta [\kappa_0 n]_d)/d&\mathrm{if}& \notmodd{n}{0}{d},\\
	\end{array}\right.
\end{equation}
where  $\alpha> d \geq2$  are two  positive integers, $\beta$ is  another integer which may be negative and $\kappa_0=\pm1$.  The integers  $d$, $\alpha$ and $\beta$  are  such that:   $\notmodd{\alpha}{0}{d}$ and  $\notmodd{\beta}{0}{d}$.   The triplet $(d,\alpha,\beta)$ associated to $T$  will be  denoted by $(d,\alpha,\beta)_{\pmb{+}}$ and 
$(d,\alpha,\beta)_{\pmb{-}}$ for $\kappa_0=+1$ and for $\kappa_0=-1$, respectively. A necessary and sufficient condition that the map $T:\mathbb{N}:\longrightarrow \mathbb{N}$  is well defined from $ \mathbb{N}$ into $\mathbb{N}$ is that the condition (see \cite{Bouhamidi_Hal2024}):
\begin{equation}\label{CdtCollatztriplet}
	\alpha+\kappa_0\beta>\dfrac{(\kappa_0-1)}{2}\quad\textrm{and}\quad \modd{\alpha+\kappa_0\beta}{0}{d},
\end{equation}
hold. 
Given $n\in \mathbb{N}$, the  trajectory of $n$ is the set $\Gamma(n)=\{n,T(n),T^{(2)}(n),\ldots\}$ of iterates starting by $n$.  A cycle (if it exists),  having $k$ elements (or vertices),  associated to the triplet $(d,\alpha,\beta)_{\pmb{\pm}}$  is a finite set  $\Omega $  for which $T^{(k)}(x)=x$ for all $x\in \Omega$.  It  will denoted by  
$\Omega(\omega)$ where $\omega$ is the smallest element in  the cycle. 
We will also use this abusive following  notation 
$$(\omega \rightarrow T(\omega)\rightarrow T^{(2)}(\omega)\rightarrow \ldots \rightarrow T^{(k-1)}(\omega)\rightarrow \omega),$$
to denote  the  cycle $\Omega(\omega)$.  We denote by $\mathrm{Cycl}(d,\alpha,\beta)_{\pmb{\pm}}$ the set of cycles associated to the triplet $(d,\alpha,\beta)_{\pmb{\pm}}$. The order of a triplet $(d,\alpha,\beta)_{\pmb{\pm}}$ is the number of its cycles, namely the cardinal  $\#\mathrm{Cycl}(d,\alpha,\beta)_{\pmb{\pm}}$ of the set $\mathrm{Cycl}(d,\alpha,\beta)_{\pmb{\pm}}$. It may be  infinite, zero or  finite integer. We have the following definition.

\begin{definition}
	Let $(d,\alpha,\beta)_{\pmb{\pm}}$  be a triplet with its associated  mapping   $T$ given by  \eqref{mapT}.  We will say that the triplet $(d,\alpha,\beta)_{\pmb{\pm}}$  is 
	\begin{itemize}
		\item  \textbf{admissible}   
		if and only if $(d,\alpha,\beta)_{\pmb{\pm}}$
		satisfies the condition  \eqref{mANDalphaBis} and it has at least one cycle. So,  $1\leq \#\mathrm{Cycl} (d,\alpha,\beta)_{\pmb{\pm}}$.
		\item   \textbf{weakly  admissible}   if and only if $(d,\alpha,\beta)_{\pmb{\pm}}$ is admissible and  has a finite number of cycles.  So,  $1\leq \#\mathrm{Cycl} (d,\alpha,\beta)_{\pmb{\pm}}<+\infty$. It may have divergent trajectories.
		\item   \textbf{strongly  admissible}  if and only if $(d,\alpha,\beta)_{\pmb{\pm}}$  is  weakly  admissible   without any divergent trajectory.  Namely, the triplet  has a finite number of   cycles without any divergent trajectory.
	\end{itemize} 
\end{definition}

The classical Collatz conjecture may now be reformulated as following:
\begin{conjecture}[of Collatz]
	The triplet $(2,3,1)_+$ is  strogly admissible  of order one.  Its unique cycle is the trivial one
	$\Omega(1)=(1\rightarrow  2\rightarrow 1)$ of length $2$. 
\end{conjecture}
We may also reformulate our proposed conjecture as 
\begin{conjecture}\label{MyConjecture_2}
	The triplet $(10,12,8)_+$ is  strogly admissible  of order one.  Its unique cycle is the trivial one
	$\Omega(4)=(4\rightarrow  8\rightarrow 16\rightarrow  24\rightarrow 32\rightarrow  40\rightarrow 4)$ of length $6$. 
\end{conjecture}

The remainder of this paper is organized as follows. In Section~\ref{sec:MainConjecture}, we formulate the main general conjecture, which encompasses the classical Collatz conjecture together with our proposed new conjecture.  
In Section~\ref{Sec:Toatal Stopping Time}, we give some formuleas  concerning the total stopping time for the   general mapping \ref{mapT}. The   Section~\ref{sec:Verification} is devoted to the verification of the conjectures together with an
algorithm for the backward operator that allows us the construction of graphs illustrating the conjecture.

\section{ The $(2^p+2^{q+1})n+2^p$
	problem  modulo $2^p+2^q$ where $p\geq q\geq0$ as an extension of the Collatz conjecture}\label{sec:MainConjecture}
\subsection{Case $d=2^qd'\geq 2$ with $d'\geq 2$ and $q\geq 0$}

Let $q\geq 0$ be a positive integer and for the general map case \eqref{mapT}, we choose   $d=2^qd'\geq 2$ with $d'\geq 2$ and $ \alpha = d+2^q$,  $\beta = d-2^q$ and $\kappa_0=1$.  Then , the corresponding mapping \eqref{mapT}  is given by
\begin{equation}\label{map_d=2^qd'}
	T(n)=\left\{\begin{array}{lcc}
		n/d&\mathrm{if}& \modd{n}{0}{d},\\
		\dfrac{	(d+2^q) n+(d-2^q) [n]_d}{d}&\mathrm{if}& \notmodd{n}{0}{d},\\
	\end{array}\right.
\end{equation}
which may be simplified for $q\geq 1$ as
\begin{equation}\label{map_case4_0BisA^q}
	T(n)=\left\{\begin{array}{lcc}
		n/d&\mathrm{if}& \modd{n}{0}{d}\\
		\dfrac{	(d'+1) n+(d'-1) [n]_d}{d'}&\mathrm{if}& \notmodd{n}{0}{d},\\
	\end{array}\right.
\end{equation}
We have the following theorem.
\begin{theorem}\label{trivialcycleTheoremD+2^q}
	Let $q\geq 0$, for $d=2^qd'\geq 2$, then  the triplet  $(d,\alpha,\beta)_{\pmb{+}}:=(d,d+2^q,d-2^q)_{\pmb{+}}$ is  admissible and it  has at least  the following
	trivial cycle of length $\dfrac{1}{2^q}d+q$:
	\begin{equation}\label{Trivialcycle(d,d+2^q,d-2^q)}
		\Omega(\frac{1}{2^q}\beta)=\bigl(\frac{1}{2^q}\beta\rightarrow
		\frac{1}{2^{q-1}}\beta\rightarrow\ldots\rightarrow\frac{1}{2^2}\beta\rightarrow\frac{1}{2}\beta\rightarrow\beta\rightarrow 2\beta\rightarrow 3\beta\rightarrow  \ldots\rightarrow d'\beta\rightarrow \frac{1}{2^q}\beta\bigr),
	\end{equation}
	for $q\geq 1$. For  $q=0$ the trivial cycle  is  reduced to	the following cycle of length $d$:
	\begin{equation}\label{Trivialcycle(d,d+1,d-1)}
	\Omega(\beta)=\bigl(\beta\rightarrow 2\beta\rightarrow 3\beta\rightarrow  \ldots\rightarrow d\beta\rightarrow \beta\bigr).
	\end{equation}

\end{theorem}
\begin{proof} For $q=0$, we have obviously $T^{(k)}(\beta)=(k+1)\beta$ for $1\leq k\leq d-1$, then $T^{(d)}(\beta)=\beta$. Now, suppose $q\geq 1$. As $d=2^qd'$, then $\alpha =2^q(d'+1)$ and $\beta =2^q(d'-1)$. For $1\leq k\leq q$, we 
	have  obviously, \\
	$T(\dfrac{1}{2^k}\beta)= \dfrac{1}{2^k}\beta\;\dfrac{\alpha+\beta}{d} $. As $\alpha+\beta=2d$, we  get 	$T(\dfrac{1}{2^k}\beta)= \dfrac{1}{2^{k-1}}\beta$. Then 	$T^{(q)}(\dfrac{1}{2^q}\beta)=\beta$. And, then 
	$T(k\beta)=\dfrac{(d+2^q)k\beta+\beta[k\beta]_d}{d}$ for $1\leq k\leq d'-1$. But $k\beta=k(d-2^q)$ and $2^q\leq 2^qk\leq 2^q(d'-1)=d-2^q<d$, then $[k\beta]_d=[-2^qk]_d =d-2^qk$. It follows that 
	$T(k\beta)=\dfrac{(d+2^q)k\beta+\beta(d-2^qk)}{d}=(k+1)\beta$. So,
	$T^{(d'-1)}(\beta)=d'\beta$. But $d'\beta=\frac{1}{2^q}\beta d$. It follows that  	$T^{(d'+q)}(\frac{1}{2^q}\beta)=\frac{1}{2^q}\beta$, which gives rise to the trivial cycle \eqref{Trivialcycle(d,d+2^q,d-2^q)}
	of length $d'+q=\dfrac{1}{2^q}d+q$.
\end{proof}

\begin{example}
	For instance, for $d=12$ and  $q=1$, we get  the admissible triplet $(12,14,10)_+$ which has the trivial cycle of length $d/2^q+q=7$ obtained by \eqref{Trivialcycle(d,d+2^q,d-2^q)}  as
	$$\Omega(5)=(5\rightarrow 10\rightarrow 20\rightarrow 30\rightarrow40\rightarrow 50\rightarrow 60\rightarrow 5).$$ 
	We may verify that the triplet $(12,14,10)_+$ has also two other cycles which are 
	$$\Omega(4)=4\rightarrow 8\rightarrow 16\rightarrow 22\rightarrow4),$$  of length $6$ and 	$\Omega(1305)=(1305\rightarrow 1530\rightarrow 1790\rightarrow 2090\rightarrow2440\rightarrow 2850\rightarrow 3330\rightarrow3890\rightarrow4540\rightarrow
	5300\rightarrow6190\rightarrow7230\rightarrow8440\rightarrow9850\rightarrow11500\rightarrow13420\rightarrow15660\rightarrow
	\rightarrow 1305)$,  of length $17$. 
	It seems that the admissible triplet$(12,14,10)_+$ has only these three cycles.
\end{example}

\subsection{Particular case where $d=2^p+2^q$ with $p\geq q\geq 0$}
Now, we introduce  the more interesting case where  $d=2^p+2^q$  with    $p\geq q$  is a positive integer. Then, 
\begin{equation}\label{parametersd_dp+dq}
	d=d_{p,q}:=2^p+2^q,\quad \alpha=\alpha_{p,q}:=d_{p,q}+2^q=2^p+2^{q+1}, \quad \beta=\beta_{p,q}:=d_{p,q}-2^q=2^p,
\end{equation}
and the 
corresponding mapping $T:=T_{p,q}$, for $p\geq q\geq 0$, is given  by
\begin{equation}\label{mapMainConjecture_dpdq}
	T_{p,q}(n)=\left\{\begin{array}{lcc}
		n/(2^p+2^q)&\mathrm{if}& \modd{n}{0}{(2^p+2^q)}\\
		\dfrac{(2^p+2^{q+1})n+2^p[ n]_{p,q}}{2^p+2^q}&\mathrm{if}& \notmodd{n}{0}{(2^p+2^q)},\\
	\end{array}\right.
\end{equation}
here $[\,\,]_{p,q}=[\,\,]_{d_{p,q}}$ stands for the remainder in the Euclidean division by $d_{p,q}=2^p+2^q$.   We observe
that the case  $p =q=0$  gives the triplet $(2,3,1)_+$  of   the classical Collatz case  \eqref{CollatzMap}   and the case $p=3$ and $q=1$ gives the triplet $(10,12,8)_+$ which is exactly the triplet of our conjecture \ref{Myconjecture}.   For  $p\geq q\geq 1$,  we may simplify the map $T_{p,q}$   to
\begin{equation}\label{mapMainConjecture_dpdq__Simplified}
	T_{p,q}(n)=\left\{\begin{array}{lcc}
		n/(2^p+2^q)&\mathrm{if}& \modd{n}{0}{(2^p+2^q)}\\
		\dfrac{(2^{p-q}+2)n+2^{p-q}[ n]_{p,q}}{2^{p-q}+1}&\mathrm{if}& \notmodd{n}{0}{(2^p+2^q)}.\\
	\end{array}\right.
\end{equation}

We now state the following main conjecture, which arises from extensive computational experiments and computer-assisted verification. This will be discussed breviely in Section~\ref{sec:Verification}. We firmly believe that the following conjecture is true. It extends the classical Collatz conjecture and encompasses our proposed Conjecture~\ref{MyConjecture_2}.

\begin{conjecture}\label{MainConjecture2p2q}
	For all   integers $p$ and $q$ with $p\geq q\geq 0$, the triplet
	$(2^p+2^q, 2^p+2^{q+1}, 2^p)_+$	  
	is  strongly  admissible.\\
	Furthermore, let  $\mathbb{E}=\{(1,0),(2,1),(2,2),(3,0),(4,0),(5,2),(6,2),(7,0)\}\subset \mathbb{N}_0\times  \mathbb{N}_0$.  
	\begin{enumerate} 
		\item If $(p,q)\not\in \mathbb{E}$, then $(2^p+2^q, 2^p+2^{q+1}, 2^p)_+$	 is
		of order one, its unique  trivial cycle is $\Omega(2^{p-q})$ of length $2^{p-q}+q+1$ given as:
		\begin{equation}\label{TrivialCycle2p2q}
			\Omega(2^{p-q})=\Bigl(2^{p-q}\rightarrow  2^{p-q+1}\rightarrow
			\cdots  \rightarrow2^{p-1}\rightarrow 
			2^{p}\rightarrow
			2\cdot2^{p}\rightarrow
			3\cdot2^{p}\rightarrow
			\ldots\rightarrow (2^{p-q}+1)\cdot2^{p}\rightarrow 2^{p-q}\Bigr).
		\end{equation}
		So, $\forall n\geq 1$,  $\exists k\geq 0$ such that $T_{p,q}^{(k)}(n)=2^{p-q}$.
		\item  If $(p,q)\in \mathbb{E}$, then:
		\begin{enumerate}
			\item  If $(p,q)\not=(5,2)$,  then 
			$(2^p+2^q, 2^p+2^{q+1}, 2^p)_+$   is  of order two.     Its first  cycle is  given by 	\eqref{TrivialCycle2p2q} and its second  cycle is  $\Omega(\omega_{p,q})$ the one starting  by $\omega_{p,q}$ given bellow.\\ 
			So,  $\forall n\geq 1$, $\exists k\geq 0$ such that $T^{(k)}_{p,q}(n)\in \{2^{p-q},\omega_{p,q}\}$. 
			\vspace{0.5cm}
			\item  If $(p,q)=(5,2)$,  then $(36,40,30)_+$ is of order three.  Its first  cycle is  given by 	\eqref{TrivialCycle2p2q}  and is $\Omega(8)=\bigl(8\rightarrow16\rightarrow32\rightarrow64\rightarrow96\rightarrow128\rightarrow160\rightarrow192\rightarrow224\rightarrow256\rightarrow288\rightarrow8\bigr)$. Its second and third cycles are
			$\Omega(\omega^{(1)}_{5,2})$ and $\Omega(\omega^{(2)}_{5,2})$ starting  by $\omega^{(1)}_{5,2}=76200$  and  $\omega^{(1)}_{5,2}=87176$  of length $70$ and  $35$, receptively.\\
			So, $\forall n\geq 1$,  $\exists k\geq 0$ such that 
			$T_{5,2}^{(k)}(n)\in\{8,76200,87176\}$.
		\end{enumerate}
	\end{enumerate}	
The cycles  $\Omega(\omega_{1,0})$, $\Omega(\omega_{2,1})$, $\Omega(\omega_{2,2})$, $\Omega(\omega_{3,0})$, $\Omega(\omega_{4,0})$, $\Omega(\omega^{(1)}_{5,2})$, $\Omega(\omega^{(2)}_{5,2})$, $\Omega(\omega_{6,2})$ and $\Omega(\omega_{7,0})$
starting by $\omega_{1,0}=14$, $\omega_{2,1}=74$,  $\omega_{2,2}=67$,  $\omega_{3,0}=280$, $\omega_{4,0}=1264$,
$\omega^{(1)}_{5,2}=76200$, $\omega^{(2)}_{5,2}=87176$, $\omega_{6,2}=1264$ and $\omega_{7,0}=3027584$, respectively and of length
$9$, $7$, $6$, $21$,  $49$, $70$, $35$, $69$, $630$  respectively are given as following:
$$\begin{array}{ll}
	\Omega(\omega_{1,0})=&
	\bigl(14\rightarrow 20\rightarrow  28\rightarrow 38\rightarrow 52\rightarrow 70\rightarrow  94 \rightarrow 126\rightarrow 42\rightarrow 14\bigr),\\
	\Omega(\omega_{2,1})=&	\bigl(74\rightarrow 100\rightarrow  136\rightarrow 184\rightarrow 248\rightarrow 332\rightarrow  444 \rightarrow 74 \bigr),\\
	\Omega(\omega_{2,2})=&
	\bigl(67\rightarrow102\rightarrow156\rightarrow236\rightarrow356\rightarrow536\rightarrow67\bigr),\\
	\Omega(\omega_{3,0})=&
	\bigl(280\rightarrow 312\rightarrow 352\rightarrow 392\rightarrow 440\rightarrow 496\rightarrow 552\rightarrow 616\rightarrow 688\rightarrow
	768\rightarrow856\rightarrow952\rightarrow\\
	& 1064\rightarrow 1184\rightarrow 1320\rightarrow 1472\rightarrow 1640\rightarrow
	1824\rightarrow 2032\rightarrow 2264 \rightarrow 2520 \rightarrow280),\\
	\Omega(\omega_{4,0})=&\bigl(
	1264\rightarrow 1344\rightarrow 1424\rightarrow 1520\rightarrow 1616\rightarrow 1712\rightarrow 1824\rightarrow 
	1936\rightarrow 2064\rightarrow\\
	&2192\rightarrow 2336\rightarrow  2480\rightarrow 2640\rightarrow 2800\rightarrow
	2976 \rightarrow 3152\rightarrow3344\rightarrow 3552\rightarrow 3776\\
	&\rightarrow 4000\rightarrow4240\rightarrow 
	4496\rightarrow
	4768\rightarrow 5056\rightarrow 5360\rightarrow 5680\rightarrow 6016\rightarrow 6384\\
	&\rightarrow6768\rightarrow 7168 \rightarrow 7600\rightarrow 8048\rightarrow 8528\rightarrow
	9040\rightarrow 9584\rightarrow 
	10160\rightarrow 10768	\\
	& \rightarrow 11408\rightarrow 12080\rightarrow 12800\rightarrow 13568\rightarrow 
	14368\rightarrow 15216\rightarrow 16112\rightarrow 17072\rightarrow	\\
	&  18080\rightarrow 19152\rightarrow
	20288\rightarrow 21488\rightarrow 1264\bigr),\\
	\Omega(\omega^{(1)}_{5,2})=&
	\bigl(76200\rightarrow84688\rightarrow94112\rightarrow104576\rightarrow116224\rightarrow129152\rightarrow143520\rightarrow159488\\
	&\rightarrow177216\rightarrow196928\rightarrow218816\rightarrow243136\rightarrow270176\rightarrow300224\rightarrow333600\rightarrow\\&370688\rightarrow411904\rightarrow457696\rightarrow508576\rightarrow565088\rightarrow627904\rightarrow697696\rightarrow775232\\&\rightarrow861376\rightarrow957088\rightarrow1063456\rightarrow1181632\rightarrow1312928\rightarrow1458816\rightarrow1620928\\&\rightarrow1801056\rightarrow2001184\rightarrow2223552\rightarrow2470624\rightarrow2745152\rightarrow3050176\rightarrow3389088\\&\rightarrow3765664\rightarrow4184096\rightarrow4649024\rightarrow5165600\rightarrow5739584\rightarrow6377344\rightarrow7085952\\&\rightarrow196832\rightarrow218720\rightarrow243040\rightarrow270048\rightarrow300064\rightarrow333408\rightarrow370464\rightarrow\\&411648\rightarrow457408\rightarrow508256\rightarrow564736\rightarrow627488\rightarrow697216\rightarrow774688\rightarrow860768\\&\rightarrow956416\rightarrow1062688\rightarrow1180768\rightarrow1311968\rightarrow1457760\rightarrow1619744\rightarrow1799744\\&\rightarrow1999744\rightarrow2221952\rightarrow2468864\rightarrow2743200\rightarrow76200\bigr),\\
\end{array}$$	
$$\begin{array}{ll}
	\Omega(\omega^{(2)}_{5,2})=&
	\bigl(87176\rightarrow96880\rightarrow107648\rightarrow119616\rightarrow132928\rightarrow147712\rightarrow164128\rightarrow182368\\& \rightarrow202656\rightarrow225184\rightarrow250208\rightarrow278016\rightarrow308928\rightarrow343264\rightarrow381408\\&\rightarrow423808\rightarrow470912\rightarrow523264\rightarrow581408\rightarrow646016\rightarrow717824\rightarrow797600\\&\rightarrow886240\rightarrow984736\rightarrow1094176\rightarrow1215776\rightarrow1350880\rightarrow1500992\rightarrow1667776\\&\rightarrow1853088\rightarrow2059008\rightarrow2287808\rightarrow2542016\rightarrow2824480\rightarrow3138336\rightarrow87176\bigr),\\
	\Omega(\omega_{6,2})=&\bigl(1264\rightarrow1376\rightarrow1472\rightarrow1600\rightarrow1728\rightarrow1856\rightarrow1984\rightarrow2112\rightarrow2240\rightarrow2432\\
	&\rightarrow2624\rightarrow2816\rightarrow3008\rightarrow3200\rightarrow3392\rightarrow3648\rightarrow3904\rightarrow4160\rightarrow4416\rightarrow4736\\&\rightarrow5056\rightarrow5376\rightarrow5696\rightarrow6080\rightarrow6464\rightarrow6848\rightarrow7296\rightarrow7744\rightarrow8256\rightarrow8768\\&\rightarrow9344\rightarrow9920\rightarrow10560\rightarrow11200\rightarrow11904\rightarrow12608\rightarrow13376\rightarrow14208\rightarrow15104\\&\rightarrow16000\rightarrow16960\rightarrow17984\rightarrow19072\rightarrow20224\rightarrow21440\rightarrow22720\rightarrow24064\rightarrow25536\\&\rightarrow27072\rightarrow28672\rightarrow30400\rightarrow32192\rightarrow34112\rightarrow36160\rightarrow38336\rightarrow40640\\&\rightarrow43072\rightarrow45632\rightarrow48320\rightarrow51200\rightarrow54272\rightarrow57472\rightarrow60864\rightarrow64448\rightarrow\\&68288\rightarrow72320\rightarrow76608\rightarrow81152\rightarrow85952\rightarrow1264\bigr),\\
	\Omega(\omega_{7,0})=&\bigl(
	3027584\rightarrow3051136\rightarrow3074816\rightarrow3098752\rightarrow 3122816\rightarrow3147136\rightarrow3171584\rightarrow\\&3196288\rightarrow3221120\rightarrow3246208\rightarrow3271424\rightarrow3296896\rightarrow3322496\rightarrow3348352\rightarrow\ldots\ldots\\&
	\ldots\rightarrow
	367160576\rightarrow370006784\rightarrow372875136\rightarrow375765760\rightarrow378678784\rightarrow381614336\rightarrow\\& 384572672\rightarrow387553920\rightarrow390558336\rightarrow3027584\bigr).
\end{array}$$		
\end{conjecture}

\begin{example}
	\begin{enumerate}
		\item  The classical Collatz conjecture is a special case of our conjecture~\ref{MainConjecture2p2q},  corresponding to $p=q=0$ with the triplet $(2,3,1)_{\pmb{+}}$.  If the conjecture is true, then for all
		$ n\in\mathbb{N}$, $\exists k\in \mathbb{N}_0$ such that  $T_{0,0}^{(k)}(n)=2^{0-0}=1$ and its unique cycle is the trivial cycle $\Omega(1)=(1\rightarrow 2\rightarrow 1)$ of length $2^{p-q}+q+1=2$. We recover the classic Collatz conjecture. 
		
		\item The case  $p=3$ and $q=1$  leads to the triplet $(10,12,8)_{\pmb{+}}$ corresponding to  our first conjecture~\ref{MyConjecture_2}.  If the previous conjecture is true, then  thr triplet $(10,12,8)_{\pmb{+}}$ is  strongly  admissible   of order one and  its nique trivial cycle is 
		$\Omega(2^{3-1})=(4\rightarrow 8\rightarrow 16\rightarrow 24\rightarrow32\rightarrow40\rightarrow4)$ of length $2^{p-q}+q+1=6$. For all $n\geq 1$, there exists an integer $k\geq 0$ such that $T^{(k)}(n)=2^{3-1}=4$, where the mapping $T$ is given by \eqref{mapT0}.
		
		\item 	For $p=2$ and $q=0$, we get   the triplet $(5,6,4)_+$. The  corresponding 
		mapping  is $T_{2,0}:\mathbb{N}\longrightarrow \mathbb{N} $ given  by
		\begin{equation}\label{mapT_Exp2}
			T_{2,0}(n)=\left\{\begin{array}{ccc}
				n/5&\mathrm{if}& \modd{n}{0}{5}\\ 
				\Bigl(6 n+4 [ n]_d\Bigr)/5&\mathrm{if}& \notmodd{n}{0}{5}.\\
			\end{array}\right.
		\end{equation}
		If the previous conjecture is true, then the triplet
		$(5,6,4)_+$  is strongly admissible  of order one. Its unique  trivial cycle is
		$\Omega(4)=(4\rightarrow 8\rightarrow 12\rightarrow 16\rightarrow 20\rightarrow 4)$,  of length $2^{p-q}+q+1==5$.  For all
		$ n\in\mathbb{N}$, $\exists k\geq 0$ such that  $T_{2,0}^{(k)}(n)=2^{2-0}=4$.  
		
		Let us examine some trajectories for this triplet. Starting from the integer $n=95$ the trajectory is 
		$$95\rightarrow 19\rightarrow 26\rightarrow 32\rightarrow 40\rightarrow 8\rightarrow 12\rightarrow 16\rightarrow 20\rightarrow 4.$$
		But, starting from the integer $n=83$ the trajectory is:
		$$	\begin{array}{l}
			83\rightarrow 102\rightarrow 124\rightarrow 152\rightarrow 184\rightarrow 224\rightarrow 272\rightarrow 328\rightarrow 396\rightarrow 476\rightarrow 572\rightarrow 688\rightarrow\\ 828\rightarrow 996\rightarrow 1196\rightarrow 1436\rightarrow 1724\rightarrow 2072\rightarrow 2488\rightarrow 2988\rightarrow 3588\rightarrow 4308\rightarrow 5172\rightarrow\\ 6208\rightarrow 7452\rightarrow 8944\rightarrow 10736\rightarrow 12884\rightarrow 15464\rightarrow
			18560\rightarrow 3712\rightarrow 4456\rightarrow 5348\rightarrow\\ 6420\rightarrow 1284\rightarrow 1544\rightarrow 1856\rightarrow 2228\rightarrow 2676\rightarrow 3212\rightarrow 3856\rightarrow 4628\rightarrow 5556\rightarrow\\ 6668\rightarrow 8004\rightarrow 9608\rightarrow 11532\rightarrow 13840\rightarrow 2768\rightarrow 3324\rightarrow 3992\rightarrow 4792\rightarrow 5752\rightarrow\\ 6904\rightarrow 8288\rightarrow 9948\rightarrow 11940\rightarrow 2388\rightarrow 2868\rightarrow 3444\rightarrow 4136\rightarrow 4964\rightarrow 5960\rightarrow 1192\rightarrow\\ 1432\rightarrow 1720\rightarrow 344\rightarrow 416\rightarrow 500\rightarrow 100\rightarrow 20\rightarrow 4.
		\end{array}
		$$
		Figure~\ref{fig_Traj183} presents the   trajectories  starting from $95$ and from $83$, respectively. 
		
		\begin{figure}[h]%
			\centering
			\begin{tabular}{cc}
				\includegraphics[width=5cm,height=5cm]{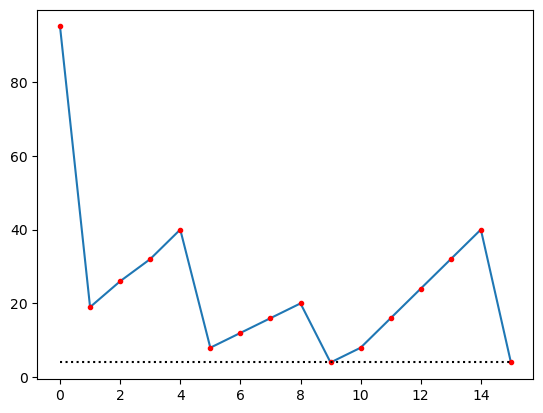}&
				\includegraphics[width=5cm,height=5cm]{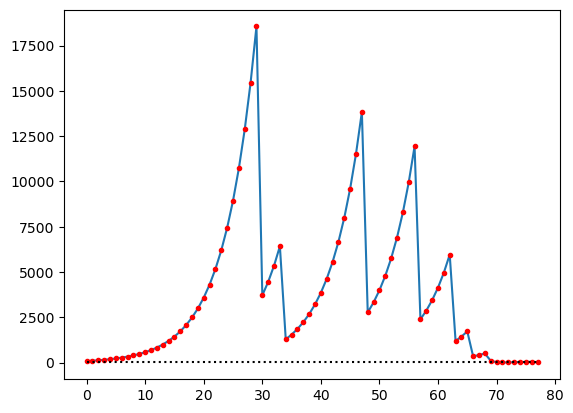}
			\end{tabular}
			\caption{Graphs  of the trajectories stating from $n=95$ (left) and from $n=83$ (right).}\label{fig_Traj183}
		\end{figure}
		
		\item 	For $p=25$~and $q=0$, we get   the triplet
		$\bigl(33445533,33445534,33445532\bigr)_+$. The  corresponding 
		mapping  $T_{25,0}:\mathbb{N}\longrightarrow \mathbb{N} $ is given  by
		\begin{equation}\label{mapT_Exp3}
			T_{25,0}(n)=\left\{\begin{array}{ccc}
				n/33554433&\mathrm{if}& \modd{n}{0}{33554433}\\ 
				\Bigl(33554434 n+33554432 [ n]_d\Bigr)/33554433&\mathrm{if}& \notmodd{n}{0}{33554433}.\\
			\end{array}\right.
		\end{equation}
		If the previous conjecture is true,  then  $\bigl(33445533,33445534,33445532\bigr)_+$ is an admissible strong triplet. Its
		unique cycle is the following trivial cycle of length $33554433$: \\
		$
		\begin{array}{ll}
			\Omega(33554432)=
			(33554432 \rightarrow 67108864\rightarrow 100663296\rightarrow 134217728 \rightarrow 167772160
			\rightarrow &\\
			201326592 \rightarrow 
			234881024\rightarrow 
			268435456\rightarrow 
			301989888\rightarrow 
			335544320\rightarrow 
			369098752\rightarrow &\\
			402653184\rightarrow 
			\ldots\ldots\ldots\ldots\ldots\ldots\ldots\rightarrow 
			1125899571298304\rightarrow 
			1125899604852736\rightarrow &\\
			1125899638407168\rightarrow 
			1125899671961600\rightarrow 
			1125899705516032\rightarrow 
			1125899739070464&\\\rightarrow 
			1125899772624896\rightarrow 
			1125899806179328\rightarrow
			1125899839733760\rightarrow&\\
			1125899873288192\rightarrow
			1125899906842624\rightarrow 1125899940397056 \rightarrow 33554432).&
		\end{array}
		$
		and for all $n\in \mathbb{N}$ there exists $k\geq 0$ such that $T_{25,0}^{(k)}(n)= 2^{25}=33554432$. We note that the maximum element of the trivial cycle $\Omega(33554432)$ is $1125899940397056$. 
		
	\end{enumerate}
\end{example}

\subsection{Particular case where $p=q\geq 0$}
Another special case derived from Conjecture~\ref{MainConjecture2p2q} is obtained by setting $p=q \geq 0$. Then the corresponding triplet for $p\geq 0$  is
$$(d_p,\alpha_p,\beta_p)_+:=(2^{p+1},3\times 2^p,2^p)_+.$$
Thus the 
corresponding mapping $T:=T_{p}$, for $p\geq 0$, is given  by
\begin{equation}\label{mapMainConjecture_dpSimple}
	T_{p}(n)=\left\{\begin{array}{lcc}
		n/2^{p+1}&\mathrm{if}& \modd{n}{0}{2^{p+1}}\\
		\dfrac{3n+[ n]_{p}}{2}&\mathrm{if}& \notmodd{n}{0}{2^{p+1}},\\
	\end{array}\right.
\end{equation}
here $[\,\,]_{p}=[\,\,]_{d_{p}}$ stands for the remainder in the Euclidean division by $d_{p}=2^{p+1}$.  It is clear that the classical Collatz mapping is obtained for $p=0$.   From Conjecture~\ref{MainConjecture2p2q},  we deduce the following special case conjecture.

\begin{conjecture}\label{MainConjecture2p}
	\begin{itemize}
		\item   For all $p\geq 0$ with $p\not=2$  the triplet  $(2^{p+1},3\times 2^{p},2^{p})_+$, associated to the mapping $T_p$ given by \eqref{mapMainConjecture_dpSimple}, is strongly  admissible  of order one, its unique  trivial cycle of length $p+2$ is:
		$$\Omega(1)=\bigl(1\rightarrow  2\rightarrow 2^2
		\rightarrow\cdots \rightarrow 
		2^{p+1}\rightarrow 1\bigr).$$
		For all integer $n\geq 1$, there exists an integer $k\geq 0$ such that $T^{(k)}_p(n)=1$.\\
		\item For $p=2$, the corresponding strongly admissible triplet $(8,12,4)_+$ is of order  $2$, its first trivial cycle is $\Omega(1)=\bigl(1\rightarrow2\rightarrow4\rightarrow8\rightarrow1\bigr)$  and its second trivial cycle is 
		$\Omega(67)=
		\bigl(67\rightarrow102\rightarrow156\rightarrow236\rightarrow356\rightarrow536\rightarrow67\bigr)$. For all $n\geq 1$, there exists an integer $k\geq 0$ such that $T_2^{(k)}(n)\in\{1,67\}$.
	\end{itemize}
\end{conjecture}

\begin{example}
	For $p=1$, we get the triplet $(d_p,\alpha_p,\beta_p)_{+} =(4,6,2)_{+}$, which gives the map $T_1$ as following:
	\begin{equation}\label{CollazConjectureV2}
		T_1(n)=\left\{\begin{array}{lcc}
			n/4&\mathrm{if}& \modd{n}{0}{4}\\
			\dfrac{3n+[ n]_{4}}{2}&\mathrm{if}& \notmodd{n}{0}{4},\\
		\end{array}\right.
		\quad\mbox{or}\quad
		T_1(n)=\left\{\begin{array}{lcc}
			n/4&\mathrm{if}& \modd{n}{0}{4}\\
			\dfrac{3n+1}{2}&\mathrm{if}& \modd{n}{1}{4},\\
			\dfrac{3n+2}{2}&\mathrm{if}& \modd{n}{2}{4},\\
			\dfrac{3n+3}{2}&\mathrm{if}& \modd{n}{3}{4},\\
		\end{array}\right.
	\end{equation}
	If the previous conjecture is true, then the triplet $(4,6,2)_+$ is  strongly admissible  of order one. Its  unique trivial cycle is $\Omega(1)=(1\rightarrow 2\rightarrow 4\rightarrow 1)$ and for all $n\geq 1$, there exists an integer $k\geq 0$ such that $T_1^{(k)}(n)=1$.
\end{example}

\section{Some  iterate and total stopping time formulas}\label{Sec:Toatal Stopping Time}
In this section we return back to the general triplet $(d,\alpha,\beta)_{\pmb{\pm}}$ with its associated map 
$T:\mathbb{N}:\longrightarrow \mathbb{N}$ defined  in   \eqref{mapT}.
For the case of the classical Collatz problem, Terras \cite{Terras1976}, defined the  notion of  total stopping  time which is an interesting tool to understand the behaviour of the iterations for the Collatz operator.  For  a general  admissible   triplet $(d,\alpha,\beta)_{\pmb{\pm}}$,  the  total stopping time of a positive integer $n$ may be also defined  by a similar property as for the classical Collatz case. So, assume that  $(d,\alpha,\beta)_{\pmb{\pm}}$ is an   admissible  triplet of order $c_0=\#Cycl(d,\alpha,\beta)_{\pmb{\pm}}>0$,  it has at least $c_0$ cycles: $\Omega(\omega_1),\ldots, \Omega(\omega_{c_0})$.  

Let $n$ be a positive integer. Thus, the total stopping time of $n$,  denoted by $\sigma_{\infty}(n)$,  is the smallest  integer $k$ such that  
$T^{(k)}(n)\in\{ \omega_1,\ldots, \omega_{c_0}\}$. If no such integer exists  the total stopping time is set  to be equal to $\infty$, i.e $\sigma_{\infty}(n)=\infty$. 

In this section, we will   state  basic formulas on forward iteration of the   associated mapping $T$.  Some relations  
concerning  the total stopping time
will also be  given.  Let us fist introduce the following functions, for all integers $n\geq 1$ and $k\geq 1$:
$$
s(n)=\left\{\begin{array}{lll}
	0&\mathrm{if}& \modd{n}{0}{d}\\
	1&\mathrm{if}& \notmodd{n}{0}{d},\\
\end{array}\right.
\quad\mathrm{and}\quad
s_k(n)=\sum_{i=0}^{k-1}s(T^{(i)}(n)).$$
The function $s_k(n)$ is the number of steps that  $\notmodd{T^{(i)}(n)}{0}{d}$  for $i=0,\ldots, k$. We set $s_0(n)=0$.

\begin{theorem}
	Let $d\geq 2$ and  let $(d,\alpha,\beta)_{\pmb{\pm}}$ be an admissible  triplet. For all integers $a\geq 1$ and $n\geq 1$, we have 
	\begin{equation}\label{T(adk+n)-casgenral0}
		T(ad^k+n)= \alpha^{s(n)}ad^{k-1}+T(n),\;\forall  k\geq 1,
	\end{equation}
	and 
	\begin{equation}\label{Tk(adk+n)-casgenral1}
		T^{(k)}(ad^k+n)= \alpha^{s_k(n)}a+T^{(k)}(n),\;\forall  k\geq 0.
	\end{equation}
	Thus, $\forall  k\geq 0$, we have
	\begin{equation}\label{sigmainf(adk+n)-casgenral}
		\sigma_\infty(ad^k+n)=\sigma_\infty\Bigl(\alpha^{s_k(n)}a+T^{(k)}(n)\Bigr)+k.
	\end{equation}
\end{theorem}
\begin{proof}
	The first relation \eqref{T(adk+n)-casgenral0} is obtained directly by disjunction of cases. If  $\modd{n}{0}{d}$, then
	$n = d\,T(n)$, it follows that  $ad^k+n=d\bigl(ad^{k-1}+T(n)\bigr)$ and  $T(ad^k+n)= ad^{k-1}+T(n)$. Else, we have   
	$$T(ad^k+n)=\dfrac{\alpha (ad^k+n)+\beta [\kappa_0 n]_d)}{d}=
	a\alpha d^{k-1}+ \dfrac{\alpha n +\beta [\kappa_0 n]_d}{d}= \alpha  a d^{k-1}+T(n).$$
	For the relation \eqref{Tk(adk+n)-casgenral1}, we proceed by induction on $k\geq 0$. For $k=0$, the relation is trivially true. For $k=1$,
	the relation is true by \eqref{T(adk+n)-casgenral0}.
	We assume that, for all $a\geq 1$  and for all $n\geq 1$,
	the formula \eqref{Tk(adk+n)-casgenral1} is true until an order $k$.   As
	$$T^{(k+1)}(ad^{k+1}+n)=T\Bigl[T^{(k)}((ad)d^{k}+n)\Bigr].$$
	By the induction hypothesis, we have $T^{(k)}((ad)d^{k}+n)=\alpha^{s_k(n)}ad+T^{(k)}(n)$. As the relation is true for $k=1$, then we have
	$$T^{(k+1)}(ad^{k+1}+n)=\alpha^{s(T^{(k)}(n))}\alpha^{s_k(n)}a+T\Bigl[T^{(k)}(n)\Bigr]=\alpha^{s(T^{(k)}(n))+s_k(n)}a+T^{(k+1)}(n)=\alpha^{s_{k+1}(n)}a+T^{(k+1)}(n),$$
	which show that the formula \eqref{Tk(adk+n)-casgenral1} is true at the order $k+1$.
	The relation 
	\eqref{sigmainf(adk+n)-casgenral} follows immediately from
	\eqref{Tk(adk+n)-casgenral1}.\hfill\end{proof}

Let us remark that the main interest of such theorem is that the relations given by the theorem  may be used to accelerate verification algorithms of the conjecture by using sieves on the numbers we have to check, see for instance \cite{Barina2021} for the verification of the classical Collatz conjecture.

As the triplet $(d,\alpha,\beta)_{\pmb{\pm}}$  satisfies the condition \eqref{CdtCollatztriplet}, then there exist a positive integer $\lambda_0>0$ such that $\notmodd{\lambda_0}{0}{d}$
and 
$\alpha +\kappa_0\beta =\lambda_0d^{\nu_0}$, where $\nu_0>0$ is the greatest power such that $d^{\nu_0}$ divides  $\alpha +\kappa_0\beta$. In the rest of this section, we will restrict our attention to the particular cases where 
the admissible  triplet $(d,\alpha,\beta)_{\pmb{\pm}}$  is such that $\lambda_0=1$, then
$\alpha+\kappa_0\beta =d^{\nu_0}$.

\begin{theorem}
	Let $d\geq 2$ and consider     an admissible triplet  $(d,\alpha,\beta)_{\pmb{\pm}}$  corresponding to   the parameter   $\beta =\kappa_0(d^{\nu_0}-\alpha)$ with $\alpha>d$.  For all integers $a\geq 1$, $k\geq 0$ and for all integer $r\in\{1,\ldots,d-1\}$ we have 
	\begin{equation}\label{T(adk+kappa0 r)-casgenral_4.2}
		T^{(k)}(ad^k+\kappa_0 r)=\left\{\begin{array}{lll}
			a\alpha^{q_0}+\kappa_0 r&\mathrm{if}& \modd{k}{0}{\nu_0},\\
			a\alpha^{q_0+1}+\kappa_0 r d^{\nu_0-r_0}&\mathrm{if}& \notmodd{k}{0}{\nu_0},\\
		\end{array}\right.
	\end{equation}
	where $ k=q_0\nu_0+r_0$ with $0\leq r_0< \nu_0$  and $ad^k+\kappa_0 r\geq 1$.  Thus, we have
	\begin{equation}\label{sigmainf(adk+kappa0 r)-casgenral}
		\sigma_\infty(ad^k+\kappa_0 r)=\left\{\begin{array}{lll}
			\sigma_\infty(a\alpha^{q_0}+\kappa_0 r)+k&\mathrm{if}& \modd{k}{0}{\nu_0},\\
			\sigma_\infty(a\alpha^{q_0+1}+\kappa_0 r d^{\nu_0-r_0})+k&\mathrm{if}& \notmodd{k}{0}{\nu_0}.\\
		\end{array}\right.
	\end{equation}
\end{theorem}
\begin{proof}
	For the relation \eqref{T(adk+kappa0 r)-casgenral_4.2}, we proceed by induction on $k\geq 0$. For $k=0$, the relation is trivially true. For $k=1$, we have 
	\begin{equation}\label{recurencek=1} 
		T(ad+\kappa_0 r)=\dfrac{\alpha(ad+\kappa_0 r)+\beta r}{d}=a \alpha +\kappa_0 r \dfrac{\alpha +\kappa_0\beta }{d}=a \alpha +\kappa_0 r d^{\nu_0-1},
	\end{equation}
	which shows that the formula is also true for $k=1$. We assume that, for all $a\geq 1$  and for all $r\in\{1,\ldots,d-1\}$,
	the formula \eqref{T(adk+kappa0 r)-casgenral_4.2}  is true until an order $k$.   As
	$$T^{(k+1)}(ad^{k+1}+\kappa_0 r)=T\Bigl[T^{(k)}((ad)d^{k}+\kappa_0 r)\Bigr].$$
	So,  by  the recurrence hypothesis at step $k$, it follows that:
	\begin{itemize}
		\item [-] If  $r_0=0$,  then 
		$T^{(k)}((ad)d^{k}+\kappa_0 r)=(ad)\alpha^{q_0}+\kappa_0 r$. According to the relation \eqref{T(adk+kappa0 r)-casgenral_4.2} at step 1, we get
		$$T^{(k+1)}(ad^{k+1}+\kappa_0 r)=T(a \alpha^{q_0} d+\kappa_0 r )=a \alpha^{q_0+1}+\kappa_0 rd^{\nu_0-1}.$$
		\item [-] If  $1\leq r_0\leq \nu_0-2$,  then 
		$T^{(k)}((ad)d^{k}+\kappa_0 r)=(ad)\alpha^{q_0+1}+\kappa_0 rd^{\nu_0-r_0}$. Thus
		$$T^{(k+1)}(ad^{k+1}+\kappa_0 r)=T\bigl((ad)\alpha^{q_0+1}+\kappa_0 rd^{\nu_0-r_0} \bigr)=a\alpha^{q_0+1}+\kappa_0 rd^{\nu_0-r_0-1}.$$
		\item [-] If  $ r_0=\nu_0-1$, then $k+1= (k_0+1)\nu_0$ and
		$$T^{(k)}((ad)d^{k}+\kappa_0 r)=ad\alpha^{q_0+1} +\kappa_0 r d^{\nu_0-(\nu_0-1)}=ad\alpha^{q_0+1} +\kappa_0 r d
		.$$
		It follows that
		$T^{(k+1)}(ad^{k+1}+\kappa_0 r)=T(ad\alpha^{q_0+1} +\kappa_0 r d )=a\alpha^{q_0+1} +\kappa_0 r .$
	\end{itemize}
	Which show that the formula \eqref{T(adk+kappa0 r)-casgenral_4.2}  is true at the order $k+1$. The relation  \eqref{sigmainf(adk+kappa0 r)-casgenral}    is an immediate consequence of the previous formulas  \eqref{T(adk+kappa0 r)-casgenral_4.2}.
	\hfill\end{proof}

We may also establish the following results.

\begin{theorem} 
	Let $d\geq 2$ and consider     an admissible  triplet  $(d,\alpha,\beta)_{\pmb{\pm}}$  corresponding to   $\alpha=d+1$ and  $\beta=-\kappa_0$  with $\kappa_0=\pm 1$. 	Then, 	for all integers $a\geq 1$, for all integer $r\in\{1,\ldots, d-1\} $ and  for all integer $k\geq 0$, we have  the following relations 
	\begin{equation}\label{T(adk+kappa0 r)-(d,d+1,1,kappa0)}
		T^{(k)}(ad^k+\kappa_0 r)=a\alpha^k+\kappa_0r,
	\end{equation} 	
	and for $k>2$, we have
	\begin{equation}\label{T(adk+kappa0 (2d-1))-(d,d+1,1,kappa0)}
		T^{(k)}(ad^k+\kappa_0 (2d-1))=\left\{\begin{array}{lll}
			a3^{k-2} +\kappa_0 &\mathrm{if}& d=2,\\
			a\alpha^{k-1}+2\kappa_0 &\mathrm{if}& d>2,\\
		\end{array}\right.
	\end{equation}
	It follows immediately  that for $k\geq 0$:
	\begin{equation}\label{sigma(adk+kappa0 r)-(d,d+1,1,kappa0)}
		\sigma_\infty(ad^k+\kappa_0 r)=\sigma_\infty(a\alpha^k+\kappa_0 r)+k,
	\end{equation} 	
	and  for $k>2$:
	\begin{equation}\label{sigma(adk+kappa0 (2d-1))-(d,d+1,1,kappa0)}
		\sigma_\infty(ad^k+\kappa_0 (2d-1))=\left\{\begin{array}{lll}
			\sigma_\infty(a3^{k-2}+\kappa_0)+k &\mathrm{if}& d=2,\\
			\sigma_\infty(a\alpha^{k-1}+2\kappa_0)+k  &\mathrm{if}& d>2.\\
		\end{array}\right.
	\end{equation}
\end{theorem}
\begin{proof}
	\begin{itemize}
		\item For the relation  \eqref{T(adk+kappa0 r)-(d,d+1,1,kappa0)}, we again proceed by induction on $k\geq 0$. For $k=0$, the relation is trivially true. For $k=1$, we have 
		$T(ad+\kappa_0 r)=\dfrac{\alpha(ad+\kappa_0 r )+\beta r}{d}=a \alpha +\kappa_0 r  \dfrac{\alpha +\kappa_0\beta }{d}$.
		As  $\alpha +\kappa_0 \beta= d$. Then
		\begin{equation}\label{recurencek=1C} 
			T(ad+\kappa_0  r)=a \alpha +\kappa_0r, 
		\end{equation}	
		which shows that the formula is also true for $k=1$. We assume that, for all $a\geq 1$  and for all $r\in\{1,\ldots,d-1\}$,
		the formula \eqref{T(adk+kappa0 r)-(d,d+1,1,kappa0)} is true until an order $k$.   So,  by  the recurrence hypothesis at step $k$, we get
		$$T^{(k+1)}(ad^{k+1}+\kappa_0 r)=T\Bigl[T^{(k)}((ad)d^{k}+\kappa_0 r )\Bigr]=T\Bigl[(ad)\alpha^{k}+\kappa_0r\Bigr].$$
		According to the relation \eqref{recurencek=1C} at step 1, we obtain
		$T^{(k+1)}(ad^{k+1}+\kappa_0 r)=a\alpha^{k+1}+\kappa_0r$.
		Which show that the formula  \eqref{T(adk+kappa0 r)-(d,d+1,1,kappa0)}  is true at the order $k+1$. The relation \eqref{sigma(adk+kappa0 r)-(d,d+1,1,kappa0)} is a consequence of the relation  \eqref{T(adk+kappa0 r)-(d,d+1,1,kappa0)} .
		
		\item For the relation  \eqref{T(adk+kappa0 (2d-1))-(d,d+1,1,kappa0)}, for $k> 2$ we have
		$$T^{(k)}(ad^k+\kappa_0 (2d-1))=T^{(k-1)}\bigl[T(ad^k+\kappa_0 (2d-1))\bigr].$$
		As  $\bigl[\kappa_0(ad^k+\kappa_0 (2d-1))\bigr]_d=d-1$ and $\alpha-1=d$, it follows 
		$$T(ad^k+\kappa_0 (2d-1))=\dfrac{\alpha(ad^k+\kappa_0 (2d-1))-\kappa_0(d-1)}{d}=\alpha ad^{k-1}+2\kappa_0d.$$
		Then,  we have
		\begin{equation}\label{R-3}
			T^{(k)}(ad^k+\kappa_0 (2d-1))=\left\{\begin{array}{lll}
				T^{(k-3)}\bigl[(3 a)2^{k-3}+\kappa_0\bigr] &\mathrm{if}& d=2,\\
				T^{(k-2)}\bigl[(\alpha a)d^{k-2}+2\kappa_0\bigr] &\mathrm{if}& d>2,\\
			\end{array}\right.
		\end{equation} 
		Then, using the relation    \eqref{T(adk+kappa0 r)-(d,d+1,1,kappa0)}, we obtain the relation \eqref{T(adk+kappa0 (2d-1))-(d,d+1,1,kappa0)}.
		The relations \eqref{sigma(adk+kappa0 r)-(d,d+1,1,kappa0)}
		and  \eqref{sigma(adk+kappa0 (2d-1))-(d,d+1,1,kappa0)}
		are  immediate consequences of the relations 
		\eqref{T(adk+kappa0 r)-(d,d+1,1,kappa0)} and
		\eqref{T(adk+kappa0 (2d-1))-(d,d+1,1,kappa0)}, respectively.
	\end{itemize}
	\hfill\end{proof}

\section{Computation and verification }\label{sec:Verification}

\subsection{Verification of the main conjecture}

\begin{algorithm}[thbp!]
	\caption{Algorithm for testing conjecture~\ref{MainConjecture2p2q}.}
	\label{Algo_T}
	\begin{algorithmic}[1]
		\State{Choose $p_{max}$ and $n_{max}$}
		\For{$p=0,\ldots,p_{max}$  }
		\For{$q=0,\ldots,p$  }
		\State $d = 2^p+2^q$; $\alpha = d+2^q=2^p+2^{q+1}$; $\beta =d-2^q=2^p$;
		\State{$S =\{2^{p-q}\}$}
		\If{$(p,q)==(1,0)$}
		\State{$S =\{2,14\}$}
		\EndIf
		\If{$(p,q)==(2,1)$}
		\State{$S =\{2,74\}$}
		\EndIf
		\If{$(p,q)==(2,2)$}
		\State{$S =\{1,67\}$}
		\EndIf
		\If{$(p,q)==(3,0)$}
		\State{$S =\{8,280\}$}
		\EndIf
		\If{$(p,q)==(4,0)$}
		\State{$S =\{16,1264\}$}
		\EndIf
		\If{$(p,q)==(5,2)$}
		\State{$S =\{8,76200,87176\}$}
		\EndIf
		\If{$(p,q)==(6,2)$}
		\State{$S =\{16,1264\}$}
		\EndIf
		\If{$(p,q)==(7,0)$}
		\State{$S =\{128,3027584\}$}
		\EndIf
		\For{$n=1,\ldots,n_{max}$   }
		\State $m=n$
		\While{$m\not\in S$ }
		\If{$\modd{m}{0}{d}$}
		\State m =m/d
		\Else 
		\State   {$m= (\alpha m+\beta [m]_d )/d$}
		\EndIf
		\EndWhile
		\EndFor
		\EndFor
		\EndFor
	\end{algorithmic}
\end{algorithm} 

In this subsection, we describe the computational verification of the main conjecture~\ref{MainConjecture2p2q}.
The verification procedure is based on the implementation of a naive algorithm analogous to Algorithm~\ref{Algo_T}.
Observe that if Conjecture~\ref{MainConjecture2p2q} were false for some pair of integers $(p,q)$ with
$p \leq p_{\max}$ and $q \leq p$, then the while-loop in Algorithm~\ref{Algo_T} would fail to terminate
for at least one initial value $n$ satisfying
$1 \leq n \leq n_{\max}$, where $p_{\max}$ and $n_{\max}$ denote prescribed computational bounds.  We have implemented and tested Algorithm~\ref{Algo_T} using several programming environments, including
\texttt{Matlab}, \texttt{Python}, \texttt{SageMath}, and \texttt{Mathematica}.
The computations were carried out on three laptop computers, allowing for parallel execution in \texttt{Matlab}.
In particular, using \texttt{SageMath}, it was possible to verify Conjecture~\ref{MainConjecture2p2q}
for all integers $p$ in the range
$0 \leq p \leq p_{\max} = 25$ and for all initial values $n$ with
$1 \leq n \leq n_{\max} = 10^7$.  

For each integer $p \in [0,p_{\max}]$, we computed the maximum total stopping time $\sigma_{\infty}(n)$ over all integers $n \in [1,n_{\max}]$ and all integer $q$ such that $0 \le q \le p$.  This maximum is denoted by
$$
M_{\infty}(p) := \max_{\substack{1 \le n \le n_{\max} \\ 0 \le q \le p}} \sigma_{\infty}(n),
$$
and is reached  for $q=0$. Table~\ref{TableStoppingTime} reports the values observed in our numerical experiments.

\begin{table}\label{TableStoppingTime}
	\centering
	\begin{tabular}{|c|c|c|c|c|c|c|c|}
		\hline
		$p$&$0$&$1$&$2$&$3$&$4$&$5$&$6$ \\             
		\hline
		$M_{\infty}(p)$&$246$&$213$&$268$&$374$&$349$&$731$&$737$ \\
		\hline
		\hline
		 $p$&$7$&$8$&$9$&$10$&$11$&$12$&$13$\\
		 \hline
		 $M_{\infty}(p)$&$818$&$1444$&$3152$&$3638$&$3639$&$4108$&$8205$\\
		 \hline
		 \hline
		 $p$&$14$&$15$&$16$&$17$&$18$&$19$&$20$\\
		 \hline
		 $M_{\infty}(p)$&$16398$&$32783$&$65552$&$131089$&$262162$&$131091$&$1048596$\\
		 \hline
		 \hline
          $p$&$21$&$22$&$23$&$24$&$25$&$---$&$---$\\
		 \hline
		 $M_{\infty}(p)$&$2097173$&$4194326$&$8388631$&$16777240$&$33554457$&$---$&$---$\\
		 \hline
		 \end{tabular}
		 \caption{ Maximum of  the Total Stopping Time $\sigma_{\infty}(n)$.}
	 \end{table}

For the specific case $(p,q) = (3,1)$, corresponding to Conjecture~\ref{MyConjecture_2} with the triplet
$(10,12,8)_+$, we pursued a more extensive verification by adopting an alternative strategy.
Rather than checking whether each trajectory reaches the terminal value $4$, we verified that for every
initial integer $n$ there exists an iteration index $k \in \mathbb{N}$ such that
$T_{3,1}^{(k)}(n) < n$.
This criterion is sufficient to guarantee convergence and is computationally more efficient than
explicitly detecting the eventual cycle.
Using this approach, we extended the computations up to $n \leq  6.5\times 10^{9}$.
The average processing rate was approximately $8.5 \times 10^4$ integers per second, and the total
wall-clock time was about $7.59 \times 10^4$ seconds, corresponding to roughly $21$ hours of computation.
These results allow us to conclude that Conjecture~\ref{MyConjecture_2} holds for all integers
$n$ up to $6.5\times 10^{9}$.

Although the above computational results provide substantial evidence in support of
Conjecture~\ref{MainConjecture2p2q}, the present verification should not be regarded as exhaustive.
Indeed, the numerical experiments are limited by the prescribed computational bounds on the parameters
$p$ and $n$, as well as by the available computational resources and the specific algorithmic strategies adopted.

In particular, while the conjecture has been validated for all integers $p \leq 25$ and for large
ranges of initial values $n$, these results do not preclude the existence of counterexamples beyond
the tested domains.
Extending the verification to larger values of $p$ and increasing the upper bound on $n$ would require
further algorithmic refinements, greater computational power, and possibly new theoretical insights
to reduce the computational complexity.

Consequently, the present study should be regarded as a preliminary step in the computational
verification of Conjecture~\ref{MainConjecture2p2q}.
For the case $p = 3$ and $q = 1$, corresponding to the conjecture modulo $10$, we have verified the
conjecture for all integers $n$ satisfying $1 \leq n \leq 6.5\times 10^{6}$, and it holds true within this range.
Naturally, this upper bound should be extended to much larger values, as was done in~\cite{Barina2025}
for the classical Collatz conjecture.
Ongoing work is focused on improving algorithmic efficiency and exploring advanced parallel and
distributed computing strategies to significantly extend the range of validated parameters.
These developments, together with additional theoretical considerations, will be the subject of
forthcoming work.

\subsection{Backward mapping and graphs}
\label{sec:graphs}
%

\begin{figure}[h]%
	\centering
	\begin{tabular}{cc}
		\includegraphics[width=8cm,height=8cm]{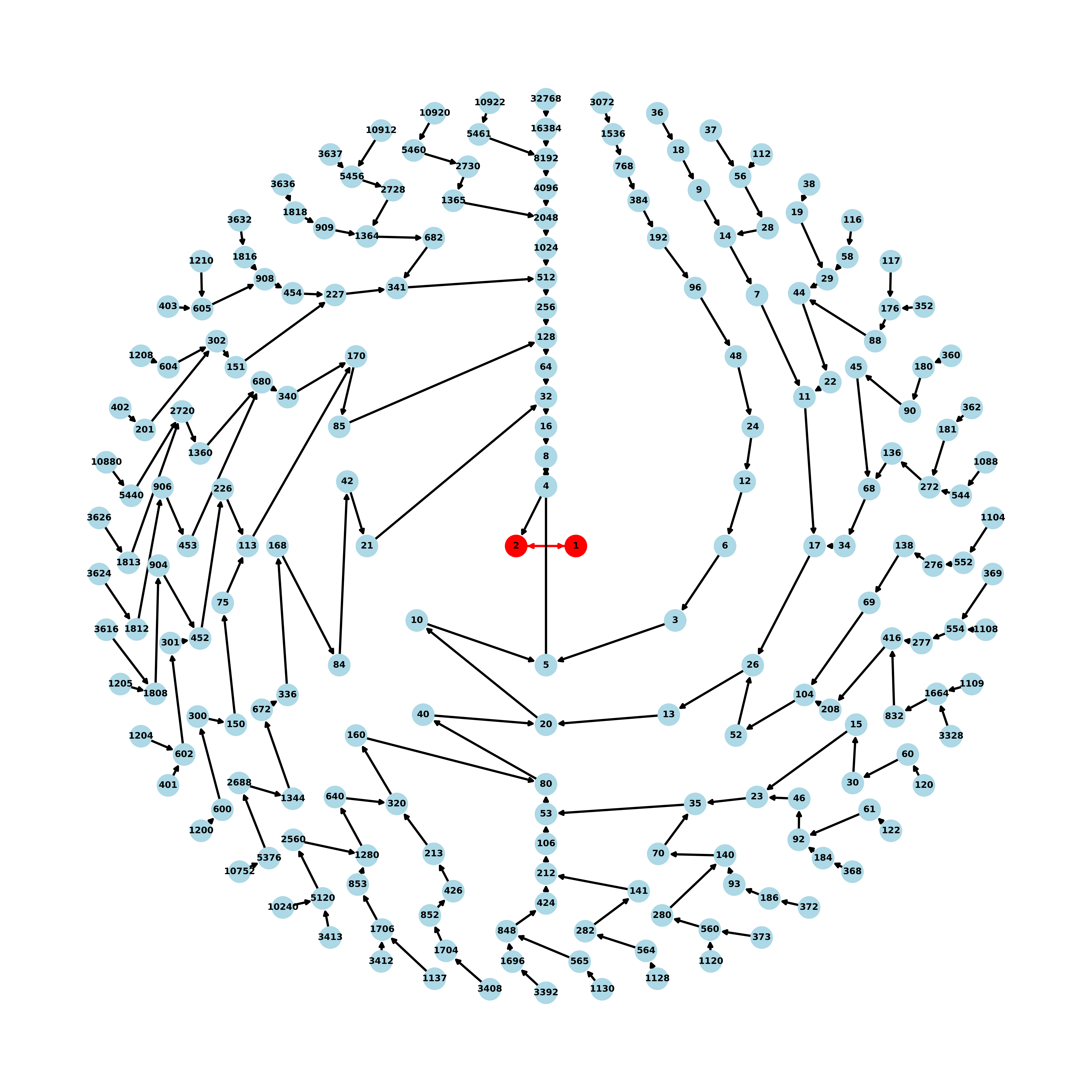}&
		\includegraphics[width=8cm,height=8cm]{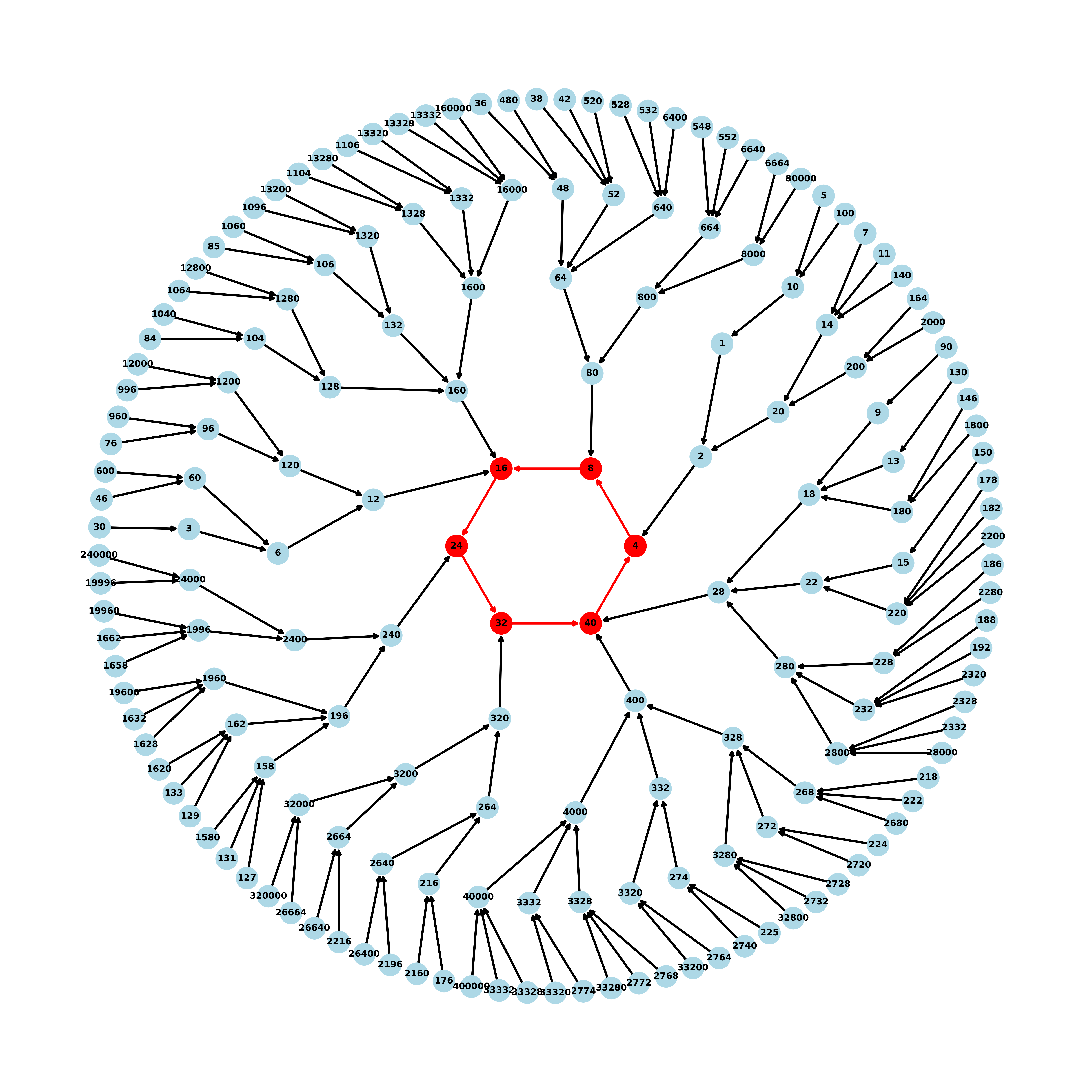}
	\end{tabular}
	\caption{Tree Graph  with its unique cycle corresponding to the triplet  $(2,3,1)_+$ (left) and to triplet $(10,12,8)_+$.}\label{fig_Graph231_10128}
\end{figure}

\begin{figure}[h]%
	\centering
	\begin{tabular}{cc}
		\includegraphics[width=8cm,height=8cm]{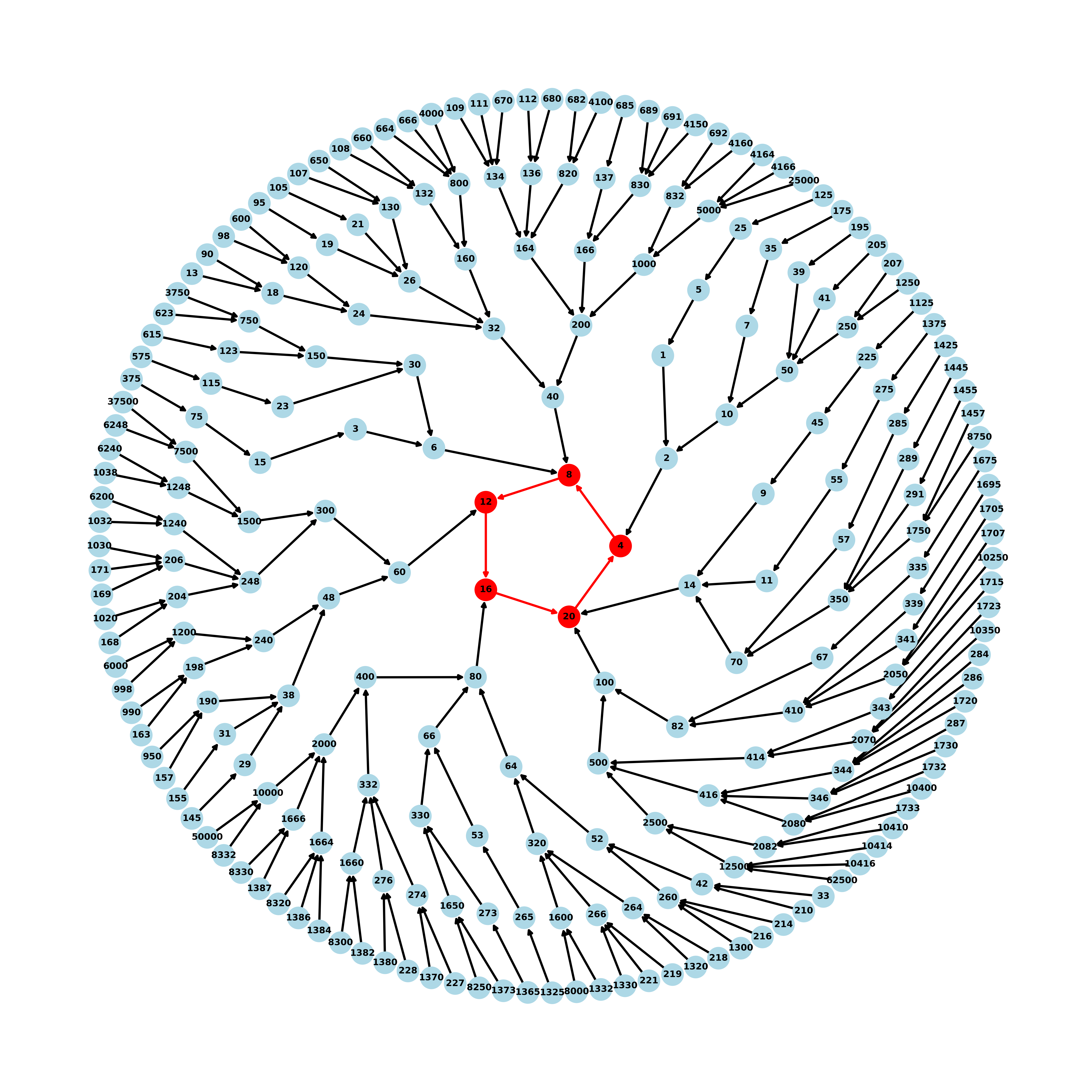}&
		\includegraphics[width=8cm,height=8cm]{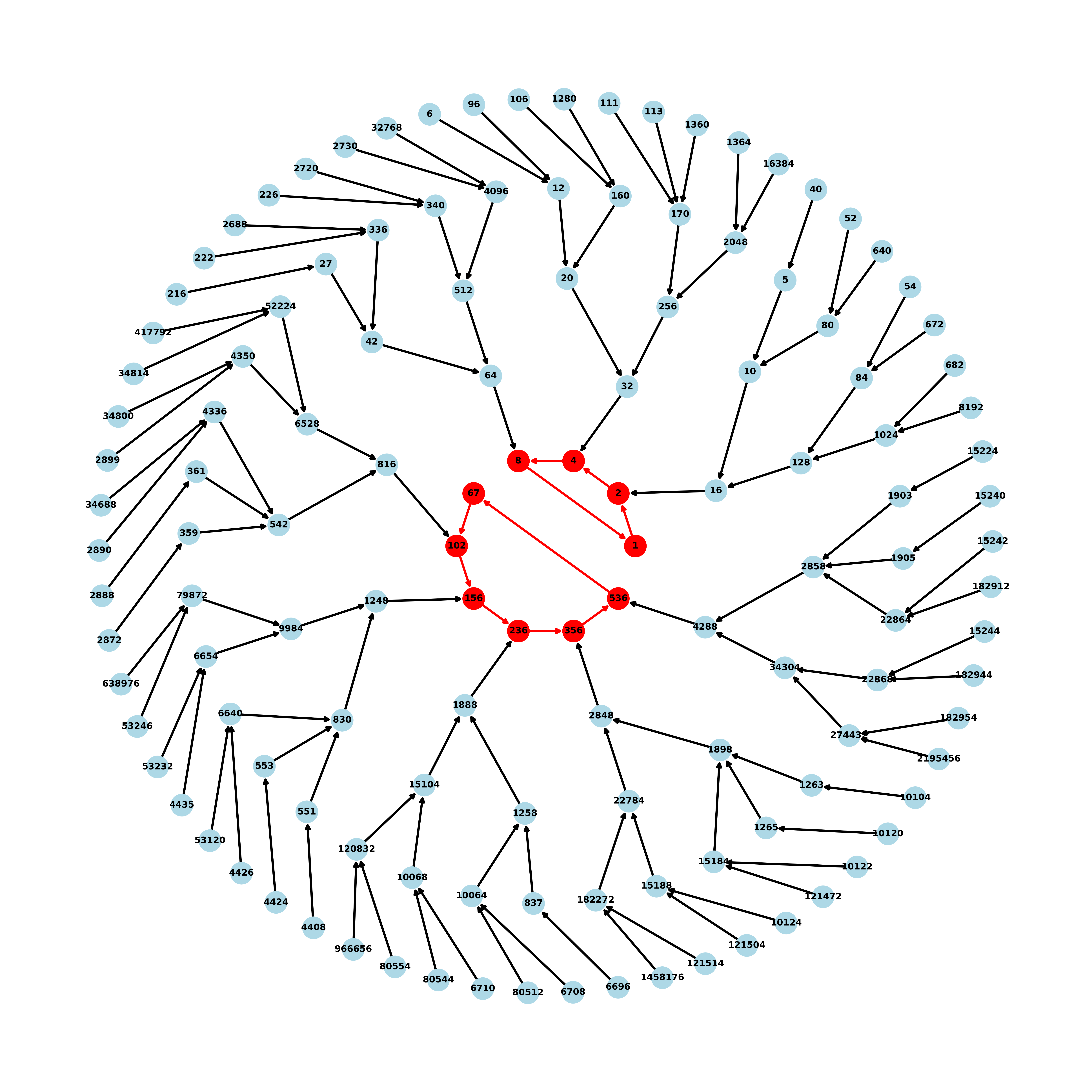}
	\end{tabular}
	\caption{Tree Graph  with its unique cycle corresponding to the triplet  $(5,6,4)_+$ (left) and to triplet $(8,12,4)_+$ with its two cycles.}\label{fig_Graph564_8124}
\end{figure}

\begin{figure}[h]%
	\centering
	\begin{tabular}{cc}
		\includegraphics[width=8cm,height=8cm]{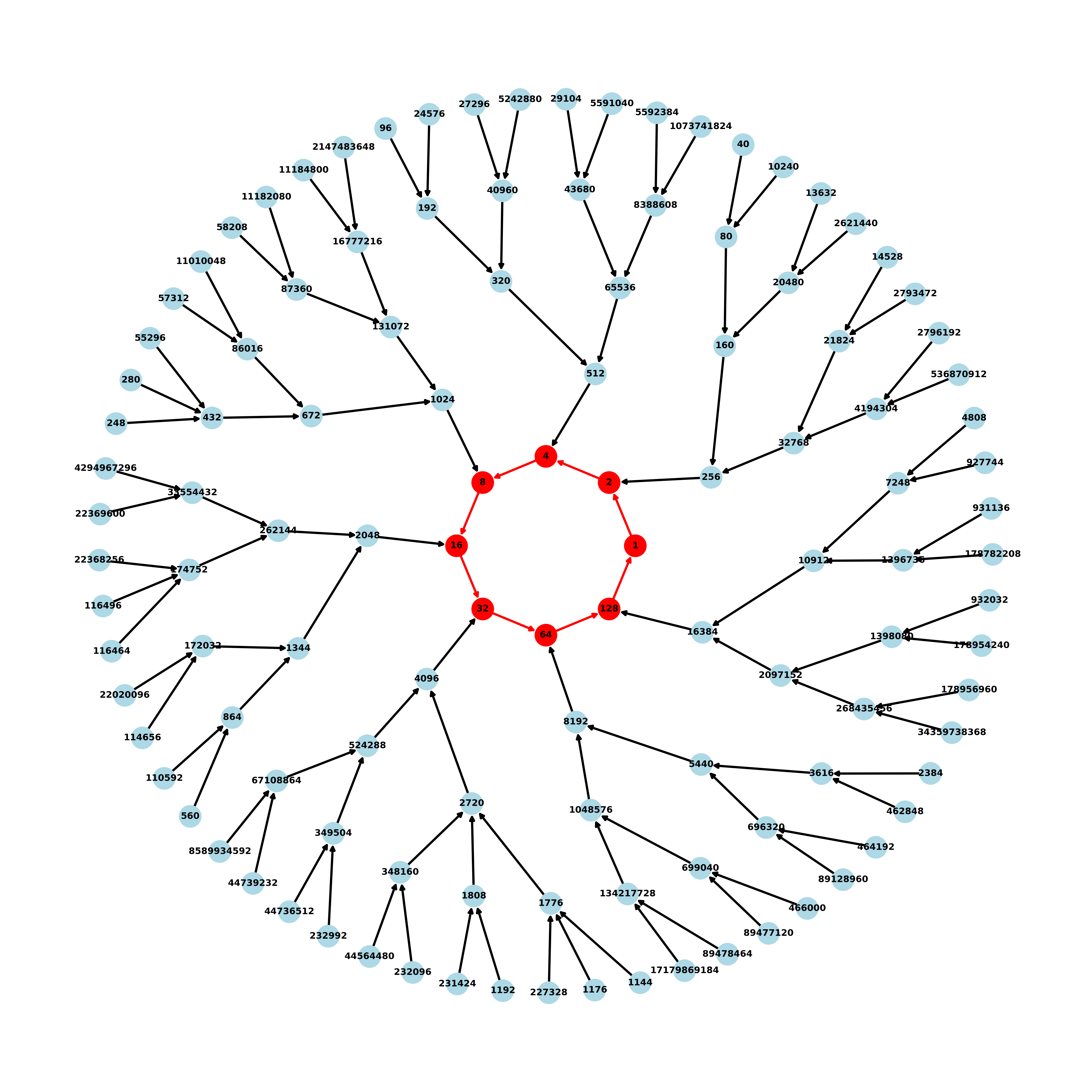}&
		\includegraphics[width=8cm,height=8cm]{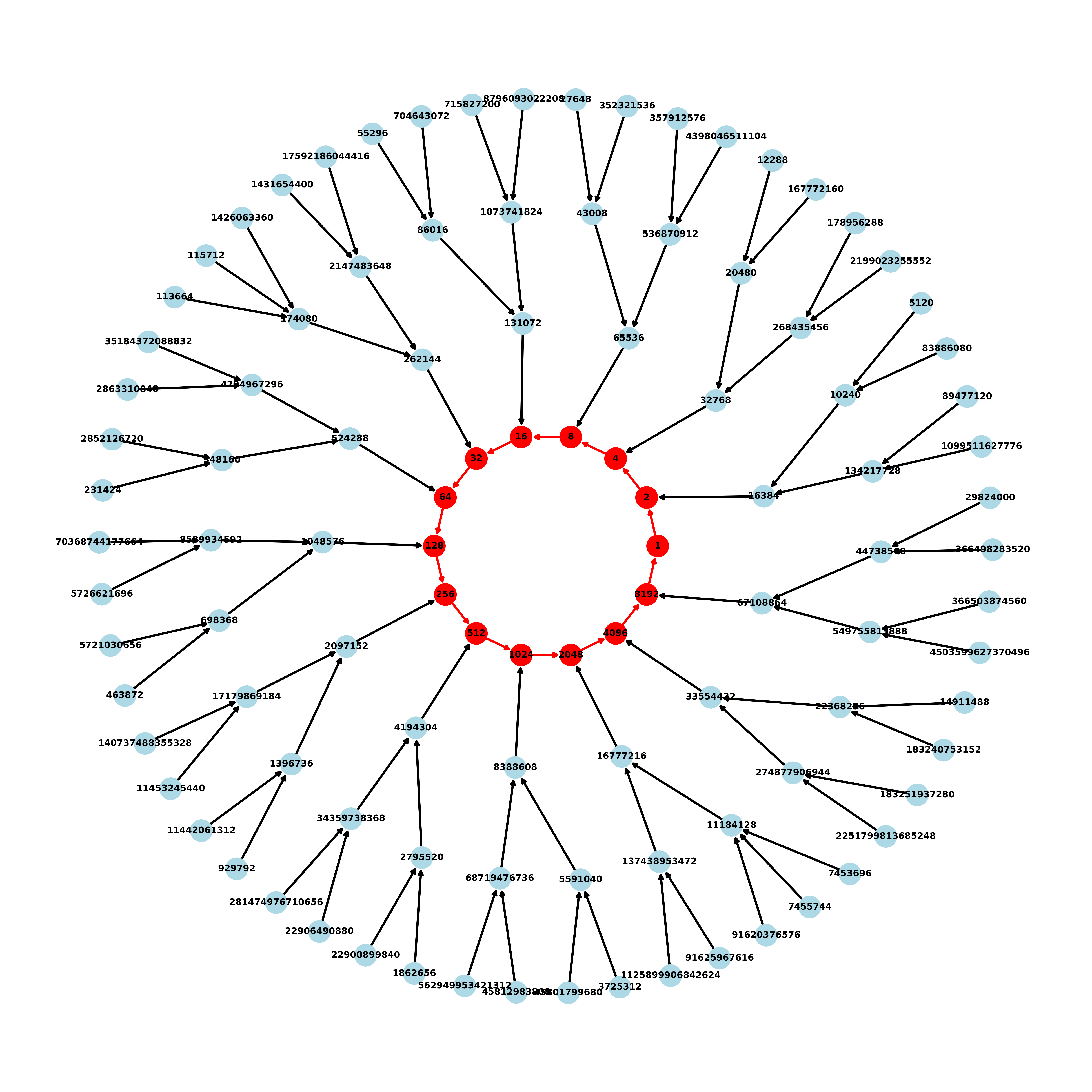}
	\end{tabular}
	\caption{Tree Graph  with its unique cycle corresponding to the triplet  $(128,192,64)_+$ (left) and to triplet $(8192,12288,4096)_+$.}\label{fig_Graph128-192-64,8192-12288-4096}
\end{figure}

\begin{algorithm}[thbp!]
	\caption{: Backward mapping of $T$ corresponding to $(d,\alpha, \beta,\kappa_0)$}
	\label{BackwardMap}
	\begin{algorithmic}[1]
		\State {\bf Input:}		An integer $n\geq 1$, The parameters $(d,\alpha, \beta,\kappa_0)$
		\State {\bf Output:}		The set  $S=\{m\in\mathbb{N}\;:\; T(m)=n\}$
		\State $S =\{d*n\}$
		\For {$r=1,\ldots,d-1$ }
		\State $R=r$
		\If {$\kappa_0=-1$}
		\State $R= d-r$
		\EndIf
		\State $m = d*n-\beta*R$
		\If{$\modd{m}{0}{\alpha}$}
		\State $m =m/\alpha$
		\If{$m>0$ and $\modd{m}{r}{d}$}
		\State   $S = S\cup \{m\}$
		\EndIf
		\EndIf
		\EndFor
		\State{\bf return} $S$
	\end{algorithmic}
\end{algorithm} 

The backward mapping provides a powerful and unifying tool for the study of Collatz-type dynamical systems. By focusing on the preimages of a given integer rather than on forward trajectories, it reveals the global structure of the dynamics and naturally leads to the construction of inverse graphs. In the general setting of Collatz maps, backward mapping is particularly useful for identifying admissible congruence classes, describing branching mechanisms, and comparing different parameterized models within a common framework. Moreover, inverse graph generation offers an effective approach for numerical exploration, allowing one to test conjectures, detect exceptional structures, and analyze the global behavior of generalized Collatz dynamics. From this perspective, the classical Collatz conjecture may be interpreted as the statement that every positive integer belongs to the backward orbit of the trivial cycle. This observation motivates the introduction of a generalized backward mapping, adapted to the map $T$ given by \eqref{mapT} and 
associated to the triplet $(d,\alpha,\beta)_{\pmb{\pm}}$.
Let $T^{-1}$ denotes the backward mapping of $T$, the  algorithm~\ref{BackwardMap} returns  the set $S=T^{-1}(n)$ for a give integer $n$ where $S$ is the set $S=\{m\in\mathbb{N}\;:\; T(m)=n\}$. By using  Algorithm~\ref{BackwardMap} it is possible to draw a directed graph  with the nodes are some integers and the  edges are $m\rightarrow n$ such that $T(m)=n$.

In figure~\ref{fig_Graph231_10128}, the graphe for  the triples $(2,3,1)_+$ and  $(10,12,8)_+$ corresponding to $(p,q)=(0,0)$ and $(p,q)=(3,1)$, respectively. In figure~\ref{fig_Graph564_81324}, the graphe for  the triples $(5,6,4)_+$ and  $(8,12,4)_+$ corresponding to $(p,q)=(2,0)$ and $(p,q)=(2,2)$, respectively.  In figure~\ref{fig_Graph128-192-64,8192-12288-4096}, the graph corresponds to the triplets
$(128,192,64)_+$  and$(8192,12288,4096)_+$ for large values of $p$ and $q$, with $(p,q)=(6,6)$ and 
$(p,q)=(12,12)$, respectively.

\section{Conclusion}
In this article, we have explored a new perspective on the Collatz problem through the introduction of an explicit example of the Collatz iteration modulo $10$. This modular framework allowed us to formulate a conjecture that unifies the classical Collatz conjecture with its modulo $10$ counterpart, while clearly identifying and characterizing the associated trivial cycles. By embedding both settings into a common structure, we highlighted the coherence between the classical dynamics and their modular analogues.
We further extended the analysis by studying the notion of total stopping time in a more general setting, providing insights into how stopping-time behavior can be examined beyond the general Collatz map. The introduction of the backward mapping may  be a powerful tool, enabling the systematic construction of directed graphs that reveal the underlying cycle structure of the iterations. These graphs offer a clear visualization of how trajectories evolve and how cycles emerge within the generalized framework.
Finally, we presented a brief computational analysis supporting the validity of the specific conjecture modulo $10$ as well as the proposed generalized conjecture. Although computational evidence cannot replace a rigorous proof, the numerical results reinforce the plausibility of the conjectures and illustrate the effectiveness of the backward approach in exploring large classes of trajectories.
Further investigations are currently in progress and will be addressed in future work. These ongoing studies aim to deepen the theoretical understanding of the proposed framework, extend the modular analysis to other bases, and push the limits of validation tests to higher values.


\end{document}